\documentclass[12pt,a4paper]{article}

\usepackage{amssymb, amsmath, amsthm}

\usepackage[cm]{fullpage}
\usepackage[english]{babel}
\usepackage[T1]{fontenc}
\usepackage[pdftex]{graphicx}
\usepackage{booktabs}
\usepackage{xcolor}
\usepackage{mathrsfs}

\usepackage{float}
\DeclareGraphicsExtensions{.png,.pdf}
\graphicspath{{./images/}}

\usepackage{hyperref}
\usepackage{autonum}

\newtheorem{thm}{Theorem}
\newtheorem{lem}{Lemma}

\title{Optimal strategy for trail running with nutrition and fatigue factors}
\author{
Bogna Jaszczak\thanks{Faculty of Pure and Applied Mathematics, Wroc{\l}aw University of Science and Technology, Poland} $^,$\footnote{Corresponding author: \texttt{bogna.jaszczak@pwr.edu.pl}}, \and
\L ukasz P\l ociniczak$^*$, 
}
\date{}

\begin{document}
\maketitle

\begin{abstract}
This paper presents an extension of Keller's classical model to address the dynamics of long-distance trail running, a sport characterized by varying terrains, changing elevations, and the critical influence of in-race nutrition uptake. The optimization of the generalized Keller's model is achieved through rigorous application of optimal control theory, specifically the Pontryagin Maximum Principle. This theoretical framework allows us to derive optimal control strategies that enhance the runner's performance, taking into account the constraints imposed by the changing terrain, nutritional dynamics, and the evolving fatigue factor. 

To validate the practical applicability of the model, simulations are performed using real-world data obtained from various mountain races. The scenarios cover various trail conditions and elevation profiles. The performance of the model is systematically evaluated against these scenarios, demonstrating its ability to capture the complexities inherent in long-distance trail running and providing valuable insight into optimal race strategies. The error in the total race-time prediction is of the order of several percent, which may give the runner a reliable tool for choosing an optimal strategy before the actual race. \\


\end{abstract}

\section{Introduction}
Competitive running, as a dynamic and complex sport, has attracted significant attention from both athletes and researchers seeking to understand and optimize performance. Keller's model \cite{keller}, originally developed to describe the dynamics of competitive running, provides a valuable foundation for studying optimal strategies in race scenarios. In this paper, we extend Keller's classical model to encompass the challenges posed by long-distance races conducted on off-road terrain, further incorporating the crucial element of nutritional consumption during the race and the fatigue factor. Long-distance races, which often take place on uneven and challenging off-road terrains, introduce additional complexities to the dynamics of competitive running. Athletes must adapt their strategies to navigate varied surfaces, inclines, and environmental conditions, prompting the need for a comprehensive modeling approach that captures these intricacies. Furthermore, nutritional considerations play a crucial role in athlete performance during extended races \cite{burke2019contemporary}. Strategic intake of nutrition can affect energy levels, endurance, and overall race outcome. Incorporating nutritional aspects into the model allows for a more realistic representation of the dynamic interaction between the runner, the terrain, and the external environment \cite{costa2019nutrition}.

The problem of finding an optimal strategy for competitive running has been widely discussed in the literature. The first mathematical models of race strategies were based on Newton's law and appeared in the pioneering works of Hill almost a century ago (see, for example, \cite{hill1928air}). Later, these ideas were thrown into the optimal control framework by Keller, who provided a solution which, when fitted to the World Records data, gave an accurate description of the finish times up to medium distance races of 10 km \cite{keller}. The original model consisted of a distance minimization problem with state equations that describe force and energy dynamics along with some algebraic constraints. Subsequently, Woodside generalized this model for longer distances by adding the fatigue term to the energy equation \cite{woodside}. Behncke \cite{behncke1987optimization} studied the optimal control problem of minimizing the time to perform the race not only for running but also for swimming. The results obtained for running were further extended for varying slope races in joint work with Andreeva \cite{book}. Maroński and Samoraj \cite{maronski} studied the impact of variable slope on the optimal velocity strategy, considering the distances of 400 and 800 meters. Mathis \cite{mathis1989effect} proposed a different approach to modeling fatigue by limiting the propulsive force available to the runner based on the previously exerted effort. In the work of Pritchard, a thorough study of the wind resistance experienced by a sprinter has been carried out \cite{pritchard1993mathematical}. Pitcher \cite{pitcher2009optimal} provided a description of the optimal strategy in the two-runner competition and gave a detailed proof of optimality. Aftalion and Bonnans \cite{opt_run_str} included variations in the volume of oxygen used per unit of time and energy recreation when slowing down. They have used an interesting idea from Morton \cite{morton} to model energy consumption and evolution using a two-container hydraulic model. Furthermore, this model was later used in \cite{aftalion2017run} to explain optimal sprinter strategies during short races. In particular, the well-known phenomenon of slowing down in the final part of the race was accurately predicted. Similarly, a recent study \cite{aftalion2019optimizing} showed how the bending of the track influences the performance of a runner. On the other hand, in \cite{cook2023optimally} authors analyzed the effect of different nutrition strategies on the outcome of the race in long-distance events such as marathons. In addition to purely strategy-oriented research, running is also actively investigated from a biomechanical point of view. From many different approaches, we mention only one of the most important: the Spring-Loaded Inverted Pendulum (SLIP) model introduced by Blickhan in \cite{blickhan1989spring}. As the name indicates, this mathematical model describes a single leg movement during the stride. Further advances were made in \cite{mcmahon1990mechanics} where the influence of leg stiffness was analyzed. The dynamics of approximate solutions to the SLIP model was investigated in \cite{geyer2005spring} which was later made rigorous in \cite{plociniczak2020solution, wroblewska2023stability, okrasinska2020asymptotic}. 

Our aim in this work is to improve the original Keller model so that it can accommodate long-distance trail races. As mentioned above, due to the varying terrain, elevation, and weather conditions, these events require a completely different strategy, preparation, and mindset of the contestant than during flat-track short- and medium-distance races. To our knowledge, this is the first mathematical approach to model such a problem (however, see also \cite{book}). Based on previous works by different authors, we supplement the original Keller model with equations related to nutrition strategy, fatigue, and varying terrain. To rigorously analyze the optimality of the generalized Keller model in the context of long distance off-road running with nutritional considerations, we employ the Pontryagin Maximum Principle \cite{liberzon2011calculus}. This enables us to derive optimal control strategies that maximize the runner's performance while accounting for the constraints imposed by the terrain and nutritional dynamics. In addition to theoretical analysis, we provide real-world data computations to validate the practical applicability of the generalized model. Using empirical data from actual long-distance off-road races, we demonstrate the effectiveness of the proposed model in capturing the details of competitive running scenarios, yielding optimal strategies that athletes can employ to enhance their performance. For exemplary routes, we chose the tracks of five races that were part of the Golden Trail World Series in 2023. In these events, elite trail runners compete in races that differ in distance, elevation gain, scenery, and surface type. We consider marathons in mountains of an alpine character (Zegama Aizkorri Maraton and Marathon du Mont-Blanc), half-marathon (Mammoth), uphill race (Pikes Peak) and skyrunning event (Dolomyths Run). Such a variety enables us to test the model on a spectrum of different trail running events. We show that the model is sufficient to describe finish times with good accuracy for all tested scenarios. Good agreement of the real-life results with the predictions shows the robustness of the proposed model. 

This paper is structured as follows. In the next section we provide a detailed derivation of the original Keller model along with a generalization. We also nondimentionalize the model to simplify further analysis. The main optimal control problem is investigated in Section 3 in which we construct a solution trajectory and prove its optimality. In Section 4 we present results of fitting the model to the real-world data and provide some examples of optimal race strategies. 

\section{Model}
In this section, we derive and prepare the model for further analysis. First, we provide a detailed derivation of the generalization of the Keller model along with a discussion of nutrition and fatigue factors. Then, we scale all the equations to facilitate further optimal control analysis.  

\subsection{Derivation}
Keller's classical competitive running model is constructed with the objective of minimizing the race time $T$ on a track of length $D$ with a correct control of the propulsive force per unit mass $f=f(t)$. Of course, these two quantities are related by
\begin{equation}
D = x(T) = \int_0^T v(t) dt,
\end{equation}
where $v=v(t)$ is the velocity of the runner. Note that in \cite{opt_run_str} it has been shown that optimal solutions of the time minimization problem are also optimal for distance maximization in a given time. It is more convenient to consider the latter problem, and we will do so in the sequel. In the following, we will describe the origin and derivation of all the state equations of our model. All quantities appearing in the mathematical formulation are summarized in Tab. \ref{tab:Model} where we also give an example of the literature reference for their value.

\begin{table}
\centering
\begin{tabular}{cccc}
    \toprule
    symbol & description & value & reference \\
    \midrule
    $D$ & race distance & 20-45 km & \\
    $T$ & race time & 1.5 - 4 h & \\
    $v$ & runner's velocity & 1.5-4.5 m/s & \\
    $f$ & propulsion force per unit mass & $6.7 m/s^2$ & \cite{book}\\ 
    $\tau$ & coefficient of the internal resistive force & 0.67 s & \cite{opt_run_str}\\ 
    $c$ & coefficient of the drag force & $3.75 \times 10^{-3}$ 1/m & \cite{behncke}\\
    $E$ & runner's total energy & & \\
    $E_0$ & initial energy level & $1.4 - 2.5 \times 10^3$ m$^2$/s$^2$ & \cite{opt_run_str}, \cite{book}, \cite{pitcher2009optimal}\\
    $\hat{\sigma}$ & physiological energy supply & $27$ m$^2$/s$^3$ & \cite{behncke} \\ 
    $m$ & runner's body mass & 65 kg & \\
    $N$ & exogenous carbohydrate oxidation rate & & \\
    $k$ & inverse of the oxidation time scale & 1.353 $1/s$ & \\
    $N_0$ & initial oxidation rate & $2 \times 10^{-3}$ g/s& \\
    $M$ & maximal oxidation rate & $2.32 \times 10^{-2}$ g/s & \cite{book} \\ 
    $\zeta$ & energy produced from each gram of oxidized carbohydrates& $1.6736 \times 10^{4}$ J/g & \\ 
    $Q$ & energy loss rate due to fatigue & & \\
    $K$ & proportionality constant in the fatigue equation & $6 \times 10^{-5} 1/s$ & \\
    $g$ & gravitational acceleration & $9.81$ m/s$^2$ & \\
    $\alpha$ & track inclination angle & $(-\frac{\pi}{2}, \frac{\pi}{2})$& \\
    \bottomrule
\end{tabular}
\caption{Physical quantities of the model.}

\label{tab:Model}
\end{table}

Equations of motion can be derived from a simple energy and force balance argument stated diagrammatically as follows
\begin{equation}
\begin{split}
    \text{net force } &= \text{ propulsion } - \text{ gravity } - \text{ internal resistive force} - \text{ drag}, \\
    \text{energy change} &= \text{ supply } - \text{ work } + \text{ nutrition } - \text{ fatigue}.
\end{split}
\end{equation} 
In contrast to the classical model, here we have included terms corresponding to gravity, nutrition, and fatigue. In the analysis of optimal strategy for running a short- to medium-distance race on a flat track, they can be safely neglected. However, for longer races, especially on varied terrain, it is crucial to take into account accumulated fatigue and energy replenishment through nutrition. Mountain marathons are good examples of the importance of all three factors, and we include them in our model. Furthermore, we also include the quadratic drag force. Finally, the internal resistive force per unit mass is usually modeled to be proportional to the velocity. The force balance leads to the following equation of motion with an initial condition
\begin{equation}\label{eqn:VelocityEq}
\frac{dv}{dt} = f - g \sin\alpha - \frac{v}{\tau} - c v^2, \quad v(0) = 0,
\end{equation}
where $g$ is the gravitational acceleration, $\alpha=\alpha(x(t))$ is a given function representing the topography of the running route, $\tau$ is the inverse of the proportionality constant for the resistive force, and $c$ is the coefficient of the quadratic drag force. We note that all forces are stated \textit{per unit mass}. This equation has to be supplemented with the one determining the runner's displacement:
\begin{equation}\label{eqn:DistanceEq}
\frac{dx}{dt} = v, \quad x(0) = 0.
\end{equation}
Similarly to the above, the energy (per unit mass) equation can be written as
\begin{equation}\label{eqn:EnergyEq}
\frac{dE}{dt} = \sigma - fv + \frac{\zeta}{m} N(t) - Q, \quad E(0) = E_0,
\end{equation}
where $\sigma$ is the energy supply by breathing and circulation, $N=N(t)$ is the rate of energy replenishment (carbohydrate oxidation rate) due to nutrition. Furthermore, the constant $\zeta$ in \eqref{eqn:EnergyEq} represents how much energy is extracted from glucose oxidation. It is known that $1g$ of carbohydrate provides $4$ kcal = $16.74$ kJ of energy. Moreover, $Q$ is another state variable representing fatigue evolving according to
\begin{equation}\label{eqn:FatigueEq}
\frac{dQ}{dt} = K fv, \quad Q(0) = 0,
\end{equation}
where similarly as \cite{woodside} we assume that the fatigue increase is proportional to the work rate done by the runner with the proportionality constant $K>0$. We treat $\sigma$ as the energetic equivalent of $\dot{V}O2$. Similarly as in \cite{opt_run_str} we assume that $1$ liter of oxygen produces $20$ kJ energy. Under this assumption, we can easily calculate the value of $\sigma$ based on the runner's $\dot{V}O2_{\max}$, usually expressed in ml kg$^{-1}$ min$^{-1}$. Therefore,
\begin{equation}
\dot{V}O2_{\max}\;\frac{\text{ml}}{\text{kg} \times \text{min}} = \frac{1}{60} \dot{V}O2_{\max} \frac{\text{ml}}{\text{kg} \times \text{s}} = \frac{1}{60000} \dot{V}O2_{\max} \frac{\text{l}}{\text{kg} \times \text{s}},
\end{equation}
and
\begin{equation}
\sigma =  \frac{1}{6\times 10^4} \; \dot{V}O2_{\max} \; \frac{\text{l}}{\text{kg} \times \text{s}} \times \; 2 \times 10^4 \;\frac{\text{J}}{\text{l}} = \frac{1}{3} \; \dot{V}O2_{\max} \; \frac{\text{J}}{\text{kg} \times \text{s}} = \frac{1}{3} \; \dot{V}O2_{\max} \; \frac{\text{m}^2}{\text{s}^3}
\end{equation}
In real life, $\sigma$ varies with time - increases at the beginning of the exercise and decreases at the end. In \cite{opt_run_str}, $\sigma$ was described as a function of $E$, relating the drop to the depleted anaerobic energy stores. In our model, treating the subject of long races, we assume that $\sigma$ is constant over time. As proposed in \cite{minetti}, we consider two limitations on $\hat{\sigma}$, representing $\dot{V}O2_{\max}$:
\begin{itemize}
\item \textbf{Duration of the race}
\\ The longer the race, the lower the value of the available fraction of $\dot{V}O2$. Following \cite{minetti} and \cite{f_d} we use the formula:
\begin{equation}
    f_d = \frac{940 - \frac{T}{60}}{1000},
    \label{eq:fd}
\end{equation}
where $T$ is the race duration expressed in seconds. Note that this equation is not suitable for the ultramarathons lasting more than 940 minutes. 

\item \textbf{Altitude above the sea level}
\\ In high altitudes, the air has low density and it is harder to provide oxygen to cells. Based on \cite{minetti} and \cite{f_a}, we define a fraction of the metabolic power available at a given altitude as:
\begin{equation}
    f_a = 1 - 11.7 \cdot 10^{-9} a^2 - 4.01 \cdot 10^{-6} a,
    \label{eq:fa}
\end{equation}
where $a$ is an altitude above the sea level expressed in meters. In our model, we will take the average altitude above sea level as a value of $a$. 
\end{itemize}
Using Eq. \ref{eq:fd} and Eq. \ref{eq:fa} we can calculate the average value of $\dot{V}O2$ available during the race, denoted by $\sigma$:
\begin{equation}\label{eqn:sigmaFract}
\sigma = \hat{\sigma} \times f_d \times f_a,
\end{equation}
where the coefficients $f_{d,a}$ correspond to the duration of the race and the inlfuence of high altitudes.

\begin{table}
\caption{Suggested carbohydrates intake}
\centering
\begin{tabular}{cc}
    \toprule                  
    Duration of exercise & Suggested amount \\  
    \midrule
    30-75 min & small amounts / mouth rinse \\
    1-2 hours & 30 g/hour \\
    2-3 hours & 60 g/hour \\
    > 3 hours & 90 g/hour \\ 
    \bottomrule 
\end{tabular}
\caption{Suggested carbohydrates intake. Source: \cite{nutrition_intake}}
\label{tab:carbs_intake} 
\end{table}

As mentioned above, nutrition is an essential factor in long-distance races. Carbohydrate ingestion during prolonged moderate to high-intensity endurance exercise has been shown to increase capacity and performance \cite{carbohydrates_impact}. The recommendations for carbohydrate intake based on \cite{nutrition_intake} are presented in the Tab. \ref{tab:carbs_intake}. In \cite{nutrition_curves} the impact of the combined ingestion of fructose or sucrose with glucose on exogenous carbohydrate oxidation rates has been studied. The authors recommend using the mix of glucose with fructose sucrose, as it reduces gastrointestinal distress and increases the capacity for exogenous carbohydrate oxidation compared to glucose alone. The exogenous rate of carbohydrate oxidation during exercise with glucose ingestion (GLU), glucose and fructose ingestion (GLU + FRU) and with glucose and sucrose ingestion (GLU + SUC) was measured in \cite{nutrition_curves} as resembled a logistic curve. This is why the energetic dynamics of the exogenous carbohydrates oxidation rate $N(t)$ can be modeled by 
\begin{equation}\label{eqn:NutritionEq}
\frac{dN}{dt} = kN\left(1-\frac{N}{M}\right), \quad N(0) = N_0,
\end{equation}
where $k$ is the inverse of the time scale of the oxidation and $M$ represents the maximal limiting value of $N$. This logistic differential equation has been chosen as the simplest model to account for the saturation and exponential nutrition intake initially after ingestion. The solution is, of course,
\begin{equation}\label{eqn:Nutrition}
N(t) = \left(\frac{1}{M}+\left(\frac{1}{N_0} - \frac{1}{M}\right)e^{-k t} \right)^{-1}.
\end{equation} 
The numerical values of the model parameters were found by fitting the logistic curve to the data from \cite{nutrition_curves} and are presented in the Tab. \ref{tab:Model}. The comparison of the fitted curve and the experimental results is presented in Fig. \ref{fig: fittedNutritionCurve}. Table \ref{tab:carbs_oxidation_Nt} presents the amount of carbohydrates oxidized according to the nutrition function used in our model.  Based on those values, we can state that providing the runner follows the suggestion of carbohydrates intake during exercise and eats a snack right before the start, the assumption of continuous oxidation is reasonable. In a recent paper \cite{cook2023optimally} authors have considered an exponential model for the oxidation rate, that is, without the inclusion of saturation. Furthermore, they have additionally considered fatty acids as a source of energy and appropriate dynamics to model their oxidation. 

\begin{figure}
\centering
\includegraphics[width=0.8\textwidth]{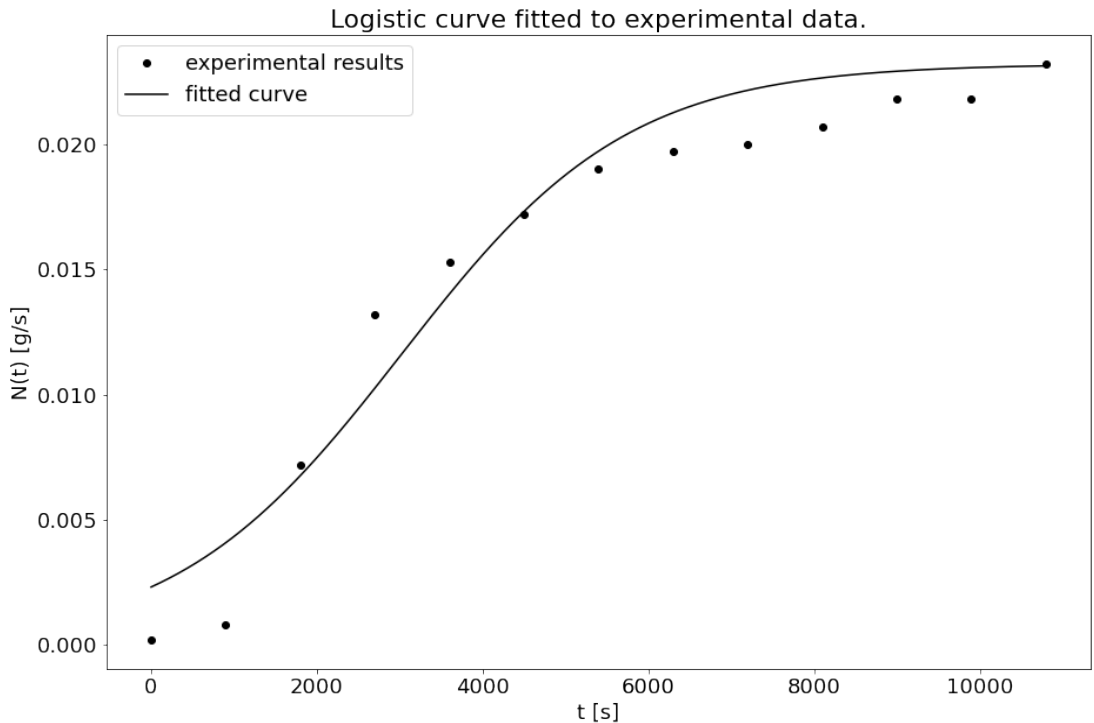}
\caption{Logistic curve fitted to experimental data. The value of the determination parameter $R^2$ is 0.9459.}
\label{fig: fittedNutritionCurve}
\end{figure}

\begin{table}
\centering
\begin{tabular}{ccc}
    \toprule                  
    Duration of exercise [h] &Amount of oxidized carbohydrates [g] & Average g/hour\\  [0.5ex] 
    \midrule
    1 & 26.1648 & 26.16 \\
    1.5 & 57.0495 & 38.03 \\
    2 & 95.1081 & 47.55\\
    3 & 177.2560 & 59.09 \\
    4 & 260.6370 & 65.16\\ [1ex]  
    \bottomrule
\end{tabular}
\caption{Amount of oxidized carbohydrates according to $N(t)$.}
\label{tab:carbs_oxidation_Nt} 
\end{table}

Finally, the model has to be closed with imposing constraints. First, the propulsive force is bounded due to physiological reasons
\begin{equation}\label{eqn:ConstraintForce}
0 \leq f \leq F,
\end{equation}  
with some $F > 0$. Second, the energy supply is also limited, hence
\begin{equation}\label{eqn:ConstraintEnergy}
0 \leq E \leq E_0.
\end{equation}
Therefore, our dynamical model of train running consists of the initial value problem \eqref{eqn:VelocityEq}, \eqref{eqn:DistanceEq}, \eqref{eqn:EnergyEq}, \eqref{eqn:FatigueEq}, supplemented with the rate of oxidation of carbohydrates \eqref{eqn:Nutrition}, along with the restrictions imposed \eqref{eqn:ConstraintForce}, \eqref{eqn:ConstraintEnergy}.

\subsection{Nondimensionalization}
\begin{table}
\centering
\begin{tabular}{ccc}
    \toprule                  
    Variable & Scale description & Symbol \\  
    \midrule
    $x$ & distance to be covered & $D = F \tau T$\\
    $v$ & velocity & $V = F\tau$ \\
    $E$ & initial energy level & $E_0$\\
    $f$ & maximal value of propulsive force & $F$ \\
    $t$ & time & $T$ \\
    $N$ & maximal oxidation rate & $M$ \\
    $Q$ & fatigue & $\Theta = KF^2\tau T$ \\
    \bottomrule 
\end{tabular}
\caption{Scales used in the nondimensionalization.}
\label{tab:scales} 
\end{table}

Table \ref{tab:scales} presents the natural scales of the variables used in the model. In addition, we need to calculate the scales for three other variables: $x$, $v$, and $Q$. We scale the velocity as 
\begin{equation}
v = Vv^{*} = F\tau v^{*}, \quad V = F \tau,
\end{equation}
where the asterisk is the new nondimensional variable. It is worth noting that in the standard Keller's model derived in \cite{keller} the quantity $F\tau$ is the limit velocity that can be achieved during the sprint races. Similarly, from \ref{eqn:FatigueEq} we obtain
\begin{equation}
Q = \Theta Q^{*} = KF^2\tau T Q^{*}, \quad \Theta = KF^2\tau T
\end{equation}
and
\begin{equation}
x = Dx^{*} = VT = F \tau T x^{*}, \quad D = F \tau T.
\end{equation}
We now can plug all nondimensional variables into our equations \eqref{eqn:VelocityEq}, \eqref{eqn:DistanceEq}, \eqref{eqn:EnergyEq}, \eqref{eqn:FatigueEq} and define the following constants
\begin{equation}
\begin{split}
    \iota = \frac{T}{\tau}, \quad \beta = \frac{gT}{F\tau}, \quad \gamma=cTF\tau, \quad \kappa = \frac{\sigma T}{E_0}, \quad \chi = \frac{F^2 \tau T}{E_0}, \\
    \phi = \frac{\zeta M T}{m E_0}, \quad \omega=\frac{KF^2 \tau T^2}{E_0},
\end{split}	
\end{equation}
so that the state equations and their constraints are now of the form
\begin{equation}\label{eqn:MainEqNondim}
\begin{cases}
    \dfrac{dv}{dt} = \iota(f - v) - \beta \sin \alpha -\gamma v^2, & v(0) = 0, \vspace{4pt}\\
    \dfrac{dx}{dt} = v, & x(0) = 0, \vspace{4pt}\\
    \dfrac{dE}{dt} = \kappa - \chi f v + \phi N - \omega Q, & E(0) = 1, \vspace{4pt}\\
    \dfrac{dQ}{dt} = f v, & Q(0) = 0, \vspace{4pt}\\ 
\end{cases}
\quad \text{with} \quad
0\leq f \leq 1, \quad 0 \leq E \leq 1,
\end{equation}
where, for greater readability, we have dropped the asterisks from the nondimensional variables. Numerical values of all appearing nondimensional constants used in our model are listed in the tab. \ref{tab:NondimConstants}. The example values taken for calculations were chosen to represent the 20 km race. A similar scaling has been used in \cite{pitcher2009optimal} for a model for a two-runner race. 
\begin{table}
\centering
\begin{tabular}{cc}
    \toprule                  
    Variable  & Value \\  
    \midrule
    $\iota$ & 8686.57 \\
    $\beta$ & 12718.69 \\
    $\gamma$ & 97.97 \\
    $\kappa$ & 62.86 \\
    $\chi$ & 70.02 \\
    $\phi$ & 13.91 \\
    $\omega$ & 24.45 \\
    \bottomrule 
\end{tabular}
\caption{Nondimensional constants used in our model.}
\label{tab:NondimConstants} 
\end{table}
As the value of $\gamma$ is much smaller than $\iota$ and $\beta$, we will neglect the effect of air resistance in the further calculations. 

\section{Optimal control}
The main problem to solve is choosing the propulsion force $f=f(t)$ in order to maximize the distance, that is the \textbf{objective}, in a given time horizon, hence
\begin{equation}\label{eqn:OptimalControl}
\text{choose the control } f \text{ such that }\int_0^1 v(t) dt \rightarrow \max,
\end{equation}
subject to the state equations with constraints \eqref{eqn:MainEqNondim}. This is an optimal control problem, and in this section, we will construct a solution and prove its optimality using Pontryagin's Maximum Principle (PMP) (for a thorough account on the optimal control theory, see \cite{liberzon2011calculus}). By Fillipov's theorem, such a problem with a convex and compact control set has a measurable solution \cite{liberzon2011calculus}. Similar problems, where the control variable appears linearly in the Hamiltonian and with pure state constraints, have been studied in \cite{maurer_linear_control}.

The starting point to solve the optimal control problem is defining the Hamiltonian augmented with the state constraint:
\begin{equation}\label{eqn:Hamiltonian}
H = \lambda_x v + \iota \lambda_v (f - v) + \lambda_E (\kappa - \chi fv + \phi N - \omega Q) + \lambda_Q(fv) + \eta E(1 - E),
\end{equation}
where $\eta$ is the penalty functions related to the state constraint and satisfies
\begin{itemize}
\item $\eta(t) \geq 0$,
\item $\eta(t) (E(t) - E^2(t)) = 0$ with $t \in [0, 1]$.
\end{itemize}
When the inequality in the corresponding constraint is strict, the penalty function is equal to zero. Otherwise, that is, on the boundary arc, we have $\eta \geq 0$. We put a multiplier related to both the upper and lower bound of $E$ as we will encounter the same behavior of the control variable $f$ in the boundary subarcs. Furthermore, the $\lambda(t)$ parameters associated with the state equations are called the \textbf{costate variables}. They are analogous to the Lagrange multipliers, but their values depend on time. A costate variable has the interpretation of being the shadow value of the state variable \cite{liberzon2011calculus}. Adjoint equations for the optimal control problem have the form
\begin{equation}\label{eqn:lambdaX}
\frac{d\lambda_x}{dt} = - \frac{\partial H}{\partial x} = \beta \lambda_v \cos \alpha \frac{d\alpha}{dx}
\end{equation}
\begin{equation}\label{eqn:lambdaV}
\frac{d\lambda_v}{dt} = -\frac{\partial H}{\partial v} = \iota \lambda_v + \chi \lambda_E f - \lambda_Q f - \lambda_x
\end{equation}
\begin{equation}\label{eqn:lambdaE}
\frac{d\lambda_E}{dt} = -\frac{\partial H}{\partial E} = \eta(2E-1)
\end{equation}
\begin{equation}\label{eqn:lambdaQ}
\frac{d\lambda_Q}{dt} = -\frac{\partial H}{\partial Q} = \omega \lambda_E,
\end{equation}
supplemented with the endpoint conditions
\begin{equation}
\lambda_x(1) = 1, \quad \lambda_v(1) = 0, \quad \lambda_E(1) = 0, \quad \lambda_Q(1) = 0.
\end{equation}
Note that changing the value of the state variables: $v$, $E$, and $Q$ at the final time will not increase the objective. Therefore, their shadow price at the final time is equal to $0$. The situation is different for $x$: our objective is to maximize its value at $t=1$, therefore, the increase in $x$ will cause a direct increase in the objective function by the same amount. We can notice that if we consider the route with the constant slope, that is, the flat route, we have $d\alpha/dx=0$ and, hence, $d\lambda_x/dt = 0$. As a result, from \eqref{eqn:lambdaX} and the endpoint conditions in such case $\lambda_x \equiv 1$ on $[0, 1]$.
For the sake of further analysis, we will point out the important properties of $\lambda_E$ and $\lambda_Q$.

\begin{lem}\label{lem:lambdas}
$\lambda_E(t)$ is positive and $\lambda_Q(t)$ is negative for $t\in [0, 1)$.
\end{lem}
The first part of the lemma follows from the monotonicity of the value function $x(T)$ with respect to $E_0$. The second part results from \eqref{eqn:lambdaQ}, the endpoint condition, and the previous observation. It is in accordance with intuition: increasing fatigue will have a negative effect on the distance covered by the runner. An additional energy supply will allow the runner to run further. Both properties can also be shown directly by integrating \eqref{eqn:lambdaE} and \eqref{eqn:lambdaQ} backward from 1. As it requires taking into account different scenarios, we will not pursue this in this paper. 

Since the Hamiltonian is linear with respect to the control variable, there exists a \textit{switching function} that has the form \cite{liberzon2011calculus}:
\begin{equation}\label{eqn:switchingFunction}
\frac{\partial H}{\partial f} = \iota \lambda_v - \chi \lambda_E v + \lambda_Q v = \psi(t).
\end{equation}
From Pontryagin's Maximum Principle we can state what happens when $\psi(t)$ is positive or negative:
$$f^* = \begin{cases}
0, & \psi(t) < 0, \\
1, & \psi(t) > 0.
\end{cases}$$
PMP, however, does not provide us with information about what happens when $\psi(t) = 0$. If the zeros of the switching functions are isolated points, they represent the switching times between the \textit{bang-bang} subarcs. When $\psi(t) \equiv 0$ on an (open) subinterval of time, we have a \textit{singular control}. We may encounter two types of singular subarcs \cite{maurer_linear_control}:
\begin{itemize}
\item \textit{boundary subarc} - occurring when the state constraint is tight on the interval. In our case, this will be related to the situations when $E \equiv 0$ or $E \equiv 1$.
\item \textit{interior subarc} - occurring when the state constraints are not tight and $0<f<1$.
\end{itemize}
To sum, up the solution may consist of the following subarcs:
\begin{equation}\label{eqn:OptimalSolution}
f^* = \begin{cases}
    0, &\psi(t) < 0, \\
    f_{int}, & \psi(t) = 0, \; 0 < E < 1, \\
    f_{b,l}, & \psi(t) = 0, \; E=0, \\
    f_{b,u}, & \psi(t) = 0, \; E=1, \\
    F, & \psi(t) > 0.
\end{cases}
\end{equation}
The constructed candidate for the solution can be used to find the optimial one as the next result states.

\begin{thm}\label{thm:optimalSolution}
Assume that the solution of \eqref{eqn:OptimalControl} subject to \eqref{eqn:MainEqNondim} consists of:
\begin{itemize}
    \item maximal force subarcs entered at the beginning of the race and on uphill fragments when $\alpha > \alpha_0$ for some $\alpha_0 > 0$,
    \item singular subarcs on the remaining parts of the route,
\end{itemize}
then it is optimal. Moreover, in the case of a flat route, that is, when $\alpha \equiv 0$ the optimal solution has precisely the following structure: the initial maximal-force subarcs, the interior arc followed by the maximal-energy subarc, again the interior subarc and finally the zero-energy subarc. 
\end{thm}
\begin{proof}
We can find the value of the control variable in the interior arc by noticing that $\psi(t) \equiv 0$ on the entire interval. We can differentiate the switching function until $f$ appears there explicitly and then solve for this variable. In our case, this means that we need to differentiate twice. As a result, we obtain:
\begin{equation}\label{eqn:fInterior}
    f_{int} = \frac{\iota^2 \lambda_x + 2 \iota \omega \lambda_E v + \beta \sin \alpha (\iota \chi \lambda_E - \iota \lambda_Q + 2 \omega \lambda_E)}{2 \iota^2 \chi \lambda_E - 2 \iota^2 \lambda_Q + \iota \omega \lambda_E}
\end{equation}
In case of the boundary arc we can notice that as when $E \equiv 0$ ($E \equiv 1)$ on the non-trivial interval of time, $dE/dt = 0$ on this interval. Using this fact, we get:
\begin{equation}\label{eqn:fBoundary}
    f_{b} = \frac{\kappa + \phi N - \omega Q}{\chi v}.
\end{equation}
We are also interested in knowing the value of the penalty function $\eta$ in the boundary subarcs. We can find its value from the fact that $\psi'(t) = 0$ there, hence
\begin{equation}
    \psi'(t) = \iota^2 \lambda_v - \iota \lambda_x + \chi \eta v - 2 \chi \eta E v + \chi \iota \lambda_E v + \chi \lambda_E \beta \sin \alpha + \omega \lambda_E v - \iota \lambda_Q v - \beta \lambda_Q \sin \alpha = 0.
\end{equation}
Collecting the $\eta$ terms yields 
\begin{equation}\label{eqn:etaValue}
    \eta = \frac{\iota \lambda_x + \iota \lambda_Q v + \lambda_Q \beta \sin \alpha - \iota^2 \lambda_v - \iota \chi \lambda_E v - \omega \lambda_E v - \chi \beta \lambda_E \sin \alpha}{\chi v - 2E\chi v}.
\end{equation}
To verify the singular arcs' optimality, we still need to check the Generalized Legendre-Clebsch condition \cite{robbins1967generalized}. As we deal with the problem of intrinsic order one, i.e. such that $q=1$ is the least integer for which
$$\frac{d^{2q}}{dt^{2q}} \frac{\partial H}{\partial f} = \frac{d^{2q}}{dt^{2q}} \psi(t)$$ 
depends explicitly upon $f$ \cite{lewis1980definitions}, GLC condition has the form:
$$\frac{\partial}{\partial f} \frac{d^2 \psi}{dt^2} \geq 0.$$
After the standard calculations we obtain:
\begin{equation}\label{eqn:GLC}
    \begin{split}
        \frac{\partial}{\partial f} \frac{d^2 \psi}{dt^2} = 2 \iota^2 (\chi \lambda_E - \lambda_Q) +  \iota \omega \lambda_E + 2 \chi^2 v^2 \eta + \chi \eta (\iota - 2 \iota E) \geq 0.
    \end{split}
\end{equation}
The condition is obviously satisfied for the interior arc, as $\iota$, $\omega$, $\chi$ and $v$ are positive, $\lambda_E \geq 0$, $\lambda_Q \leq 0$ and $\eta = 0.$  Similarly, all the summands are non-negative when $E \equiv 0$. However, we need to check what happens for the boundary subarc related to $E \equiv 1$. We can calculate the value of $\chi \eta (\iota - 2 \iota E)$ from Equation \eqref{eqn:etaValue} so that GLC reduces to:
\begin{equation}
    2 \chi^2 v^2 \eta + \frac{\iota^2}{v} \lambda_x -\frac{1}{v} \iota \beta  (\chi \lambda_E - \lambda_Q) \sin \alpha \geq 0.
\end{equation}
From the assumed form of the solution, this subarc is entered only for $-\frac{\pi}{2}<\alpha \leq 0 < \alpha_0$ such that the last summand is negative. What is important, the GLC condition is manifestly satisfied for the flat route case, when $\alpha \equiv 0.$	
\end{proof}

An exemplary illustration of the structure of the optimal solution is depicted in Fig. \ref{fig:summaryFlat} where we present the trajectories of velocity, energy, propulsive force, and fatigue for a 20 km run ($T \approx$ 1 h 37 min). Initially, the maximal force is used. Then the small drop in the energy level caused by the previous stage is quickly regenerated - it corresponds to the interior subarc. The majority of the race consists of the boundary subarc, when the maximal energy level is kept. The velocity is almost constant, but slightly decreasing. A few kilometers away from the finish line, the energy is gradually used up. It results in runner's acceleration. Finally, there is a short phase when the energy level is depleted ($E\equiv0$), but the runner still needs to finish the race. We can observe a significant drop in velocity. What is interesting, the fatigue is almost linear. The optimal strategy is then to approximately maintain the constant power throughout the race. 
\begin{figure}
\centering
\includegraphics[width=\textwidth]{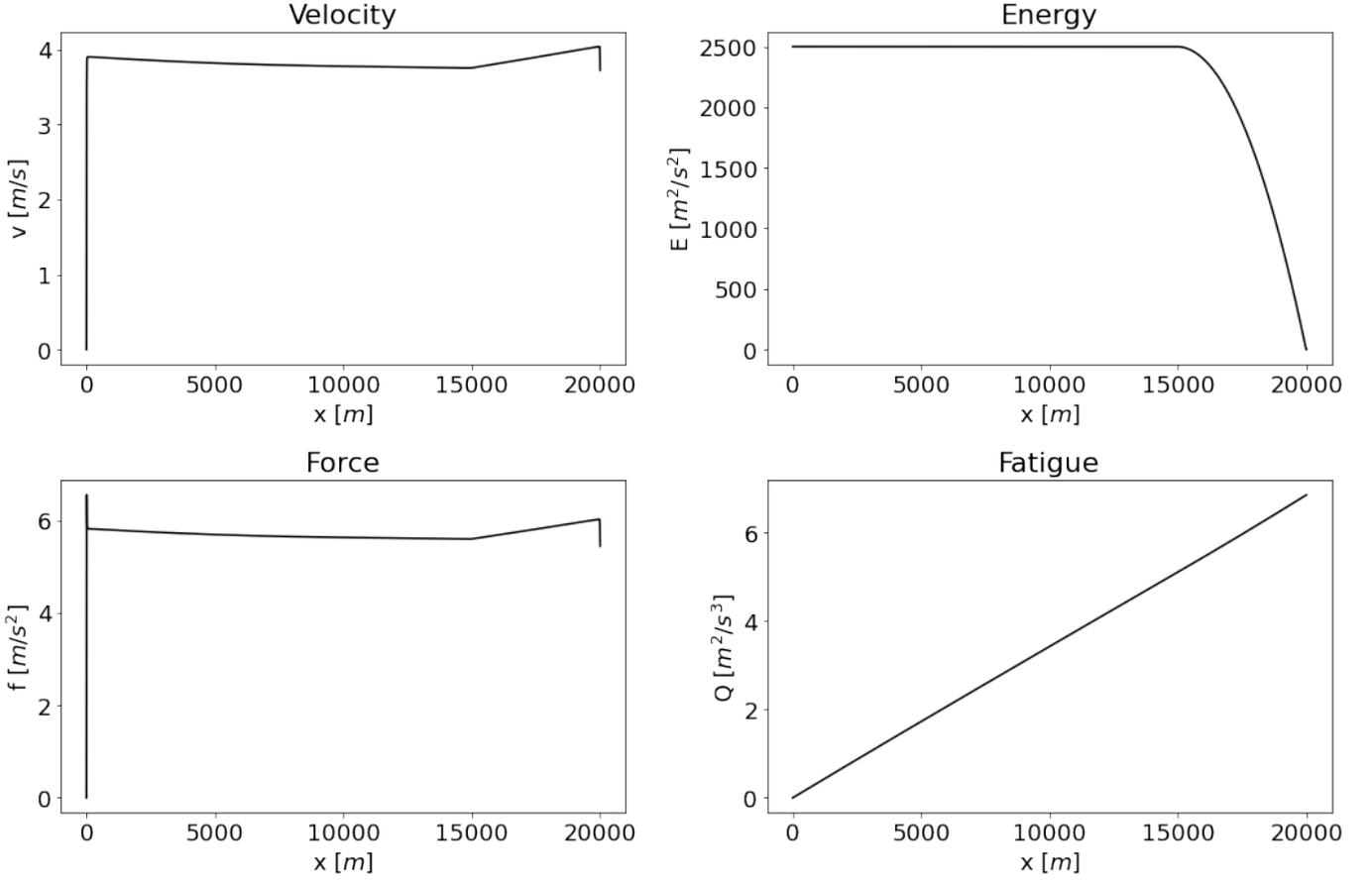}
\caption{Exemplary values of model variables for a 20 km race on a flat route.}
\label{fig:summaryFlat}
\end{figure}

\section{Real-world examples}
In real life, before the race, the runner knows only the distance and elevation profile. In case of competitions located in mountainous terrain, it is often hard to estimate the time of completing the route, especially when it is the first edition of given race and there is no possibility of comparison with previous-year results. Knowing the approximated time of completing the race is therefore important for different aspects of race logistics - especially planning the nutrition strategy. 

The optimal strategy to maximize the distance covered in a given time is the same as for the problem of minimizing the time required to run the given distance \cite{opt_run_str}. We can then utilize our model to approximate the time needed to finish the race, knowing the runner's $\dot{V}O2_{\max}$ and the route characteristic. In this section we will study the results for elite male runners during well-known mountain races, that were a part of Golden Trail World Series in 2023. In this paper we study only the men's results, however after adjusting the values of $F$, $E_0$, $m$ and $K$, the model can be successfully applied for females. 

We formulate our problem as:
\begin{equation}
\text{choose the control } f \text{ such that } \int_0^{T} dt \rightarrow \min
\end{equation}
subject to state equations \eqref{eqn:VelocityEq}, \eqref{eqn:EnergyEq}, \eqref{eqn:FatigueEq}, \eqref{eqn:distanceModified}, constraints \eqref{eqn:ConstraintForce}, \eqref{eqn:ConstraintEnergy} and the endpoint condition:
\begin{equation}\label{eqn:finalTimeXValue}
x(T) = D.
\end{equation}

\subsection{Data and parameters}
The elevation profile data was obtained from the \texttt{Strava} application. Using Strava API we were able to collect route characteristics from registered activities in the format of \texttt{.gpx} files. To reduce the complexity of numerical calculations, we averaged the slope value over short segments of $100-250$ m. It is a reasonable approach, as we consider only long distances: from 20 to 55 km, and this averaging does not lead to any loss of the overall characteristic of the route. Additionally, as \texttt{.gpx} files contain information about the horizontal distance covered, we replace \eqref{eqn:DistanceEq} with
\begin{equation}\label{eqn:distanceModified}
\frac{dx}{dt} = v \cos \alpha, \quad x(0)=0.
\end{equation}
The value of $T$ used in \eqref{eq:fd} during calculation of $\sigma$ was taken to be the current route record. Moreover, for $E_0$ we used $2 \times 10^3$ m/s$^2$.

\begin{table}
\centering
\begin{tabular}{cccc}
    \toprule                  
    Race & Distance [km] & Elevation gain/loss [m] & Route record \\  [0.5ex] 
    \midrule
    Zegama Aizkorri Maratón$^{*}$ & 41.5 & 2400/2400 & 3 h 36 min 40 s \\ 
    Marathon du Mont Blanc & 43.0 & 2500/2500 & 3 h 35 min 4 s \\
    Dolomyths Run & 21.0 & 1800/1800 & 1 h 51 min 36 s \\
    Pikes Peak Ascent & 20.5 & 2300/0 & 2h 0 min 20 s  \\
    Mammoth 26k & 27.5 & 1350/1350 & 1 h 54 min 48 s\\
    [1ex]  
    \bottomrule
\end{tabular}
\caption{Course characteristics.}
\label{tab:routeDetails} 
\end{table}

Table \ref{tab:routeDetails} presents the basic information about the races considered, such as distance, elevation gain and loss, and current route record. The information were gathered from the webpage \texttt{ratemytrail}. In case of some races, modifications are made to the routes so that the distance and elevation profile varies from year to year. In our case, this race is Zegama Aizkorri Maratón, denoted by $^*$. We decided to use the route from the year 2023. It is worth to mention that such a choice of running events yields a mixture of races of different kind: marathons, half-marathons, and one uphill. Therefore, we are able to test our model on a variety of different running scenarios. 

\subsection{Results}
All numerical calculations were performed using the GEKKO optimization suite \cite{gekko2018} \cite{gekko2022}, using IPOPT solver. The comparison of the obtained results with the current records is presented in the Tab. \ref{tab:calculatedRealComparison}. 

\begin{table}
\centering
\begin{tabular}{cccc}
    \toprule                  
    Race & Route record & Numerical result & Relative error \\  [0.5ex] 
    \midrule
    Zegama Aizkorri Maratón$^{*}$ &3 h 36 min 40 s & 3 h 35 min 34 s & 0.51 \%  \\ 
    Marathon du Mont Blanc & 3 h 35 min 4 s & 3 h 42 min 29 s & 3.45 \%  \\
    Dolomyths Run & 1 h 51 min 36 s & 1 h 40 min 24 s & 10.03 \%  \\
    Pikes Peak Ascent &  2h 0 min 20 s & 1 h 48 min 23 s & 9.93 \% \\
    Mammoth 26k &  1 h 54 min 48 s &  2 h 10 min 4 s & 13.3 \% \\
    [1ex]  
    \bottomrule
\end{tabular}
\caption{Comparison of route records and calculated finish times.}
\label{tab:calculatedRealComparison} 
\end{table}

Compared to flat-road competitions, such as athletic events or road races, it is much harder to estimate the finish time in trail races. In natural terrain, unexpected obstacles can occur, such as fallen trees, mud, or slipping rocks. Also, the weather plays a very significant role - rain or snow might cause severe or even extreme conditions, especially on the steep parts, both while running uphill and downhill. Additionally, the type of surface plays an important role; it is much easier to run on the gravel road than on the pavement consisting of large, irregular rocks. It is then hard to compare two route segments, even with the same slope, as depending on the surface type, the amount of time needed to accomplish them might be significantly different. 

With all that in mind, we can see that there is still good agreement between the real-life results and the finish times returned by the model. The estimated finish times are very reasonable in all the cases and might serve as a good guide for the runner. The parameters used for calculation lead to the most accurate results for the marathons taking place in the mountains of the alpine character, namely Zagama Aizkorri Maraton (Basque Mountains) and Marathon du Mont Blanc (Alps). Relative errors for these races do not exceed 5\%. The difference between the result returned by the model and the route record is greater in the remaining three cases. The prediction of Dolomyths Run is more optimistic than the real-life result. This can be explained with the very difficult conditions on the route: there are many switchbacks climbing up the mountains covered with the loose rocks. Such terrain prevents from running as fast as we could expect just by looking at the elevation profile. On the other hand, the real-life result of the Mammoth 26k race is better than the predicted one. Relatively small elevation gain and gravel roads make this route more convenient for running faster than the others listed. 

Considering the diversity of considered races, differing in distance, elevation gain, and surface type, as well as using generic physiological parameters in calculations, the finish times predicted by our model are very reasonable. The accuracy achieved is very good for the sport so strongly dependent on the weather and natural conditions. Additionally, we need to recall that the calculations are performed on the smoothed slope profile, with the data itself being burdened with measurement uncertainty. 

\begin{figure}
\centering
\includegraphics[width=0.8\textwidth]{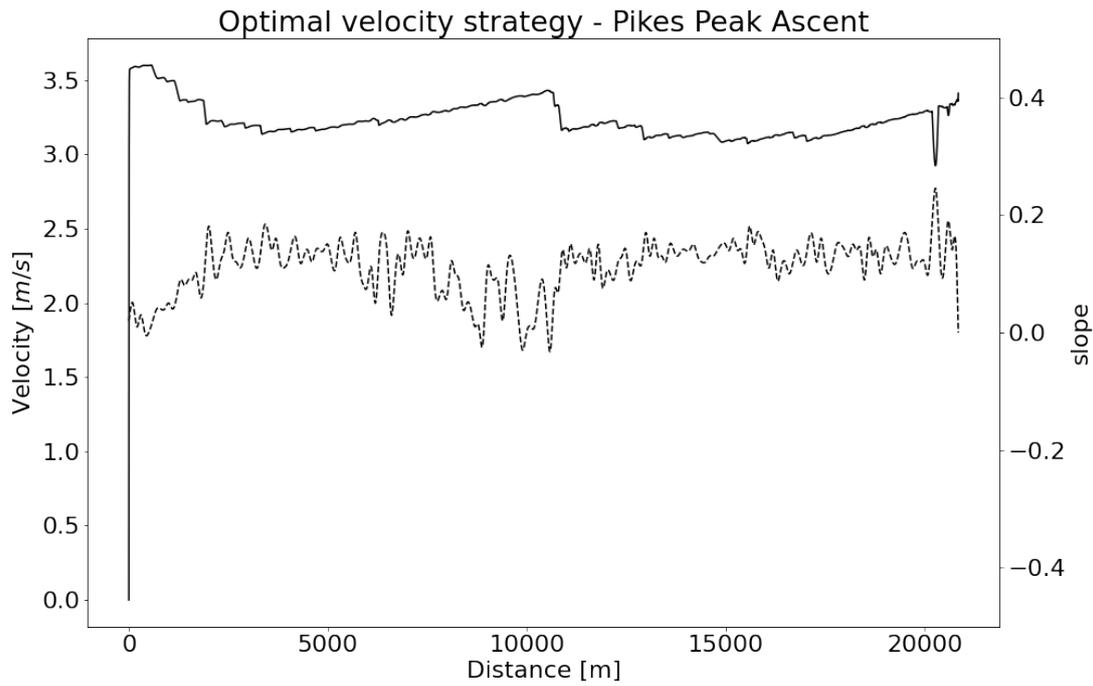}
\caption{Optimal velocity strategy for Pikes Peak Ascent race. Dashed line represents the slope expressed as $\tan \alpha$.}
\label{fig:vPikesPeak}
\end{figure}

\begin{figure}
\centering
\includegraphics[width=0.8\textwidth]{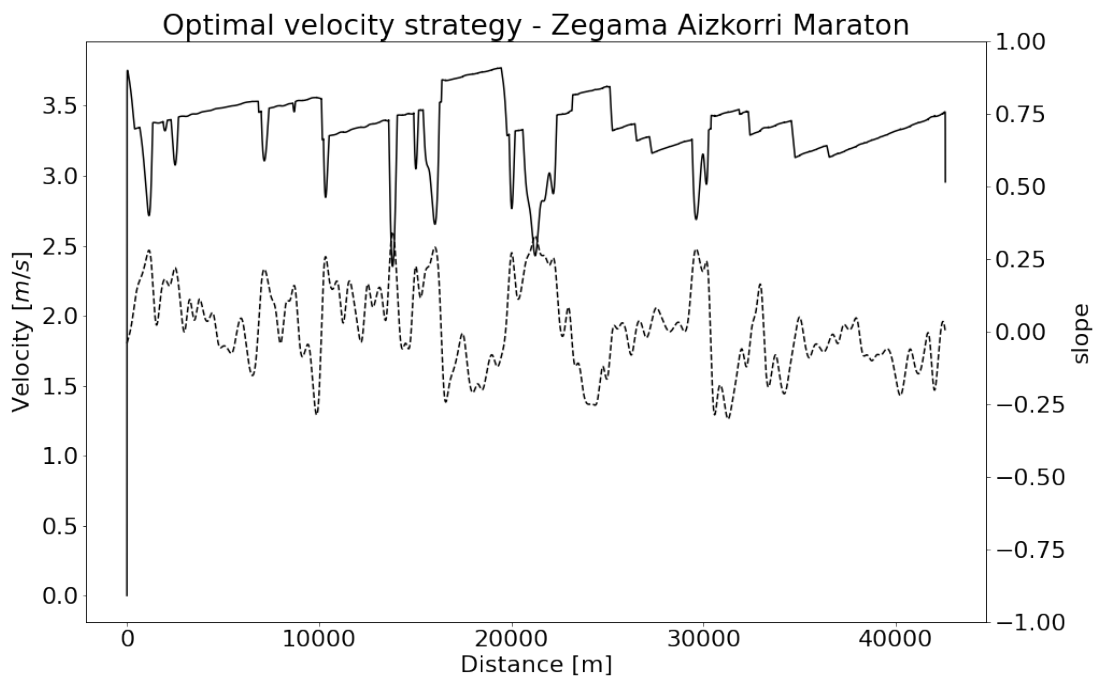}
\caption{Optimal velocity strategy for Zegama Aizkorri Maraton race. Dashed line represents the slope expressed as $\tan \alpha$.}
\label{fig:vZegama}
\end{figure}

Figures \ref{fig:vPikesPeak} and \ref{fig:vZegama} present the velocity strategies returned by the model for two particular routes. The effect of terrain shape is clearly visible: the smaller the slope variability, the more uniform the speed recommended for the race. What is interesting is that we can see that similar to the flat route there is an acceleration phase in the final part. 

In Fig. \ref{fig:allFromPikesPeak} and \ref{fig:allFromZegama} the values of state variables and the control variable obtained respectively for Pikes Peak Ascent and Zegama Aizkorri Maraton. We can clearly see that the maximal force subarc is entered at the beginning of the race and for relatively large positive values of the slope; otherwise, interior and boundary subarcs are observed. Note that this behaviour was proved in Theorem \ref{thm:optimalSolution} above. It is interesting that the energy profile in Fig. \ref{fig:allFromPikesPeak} strongly resembles that obtained for the flat route (see: Fig. \ref{fig:summaryFlat}). On the other hand, when the slope variability is higher and negative values also appear, the energy values oscillate between zero and the maximal value. What is common for these two very different races is the almost linear graph of fatigue. No matter the distance and slope profile, the optimal strategy appears to be to try to keep as uniform a power output as possible. 

\begin{figure}
\centering
\includegraphics[width=\textwidth]{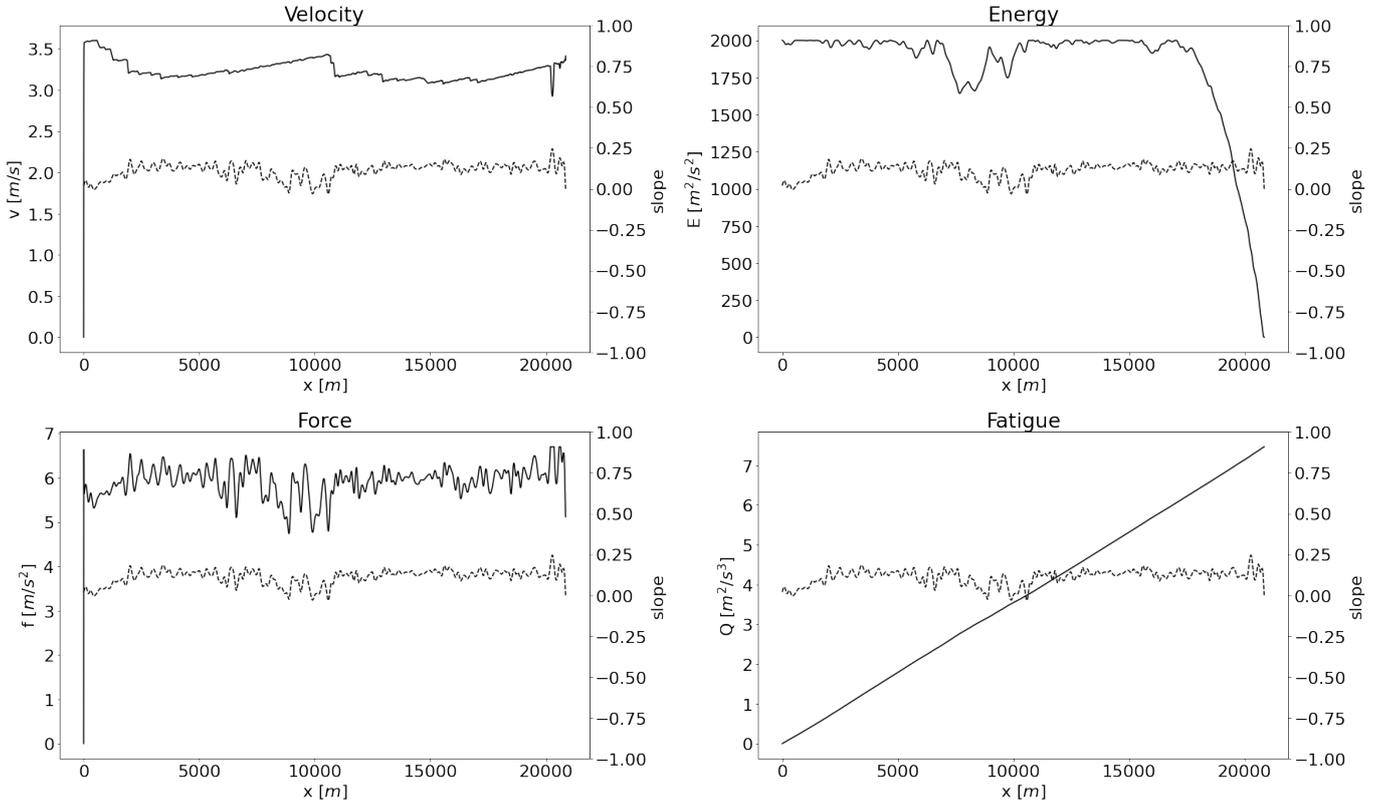}
\caption{Values of model variables for Pikes Peak Ascent.}
\label{fig:allFromPikesPeak}
\end{figure}

\begin{figure}
\centering
\includegraphics[width=\textwidth]{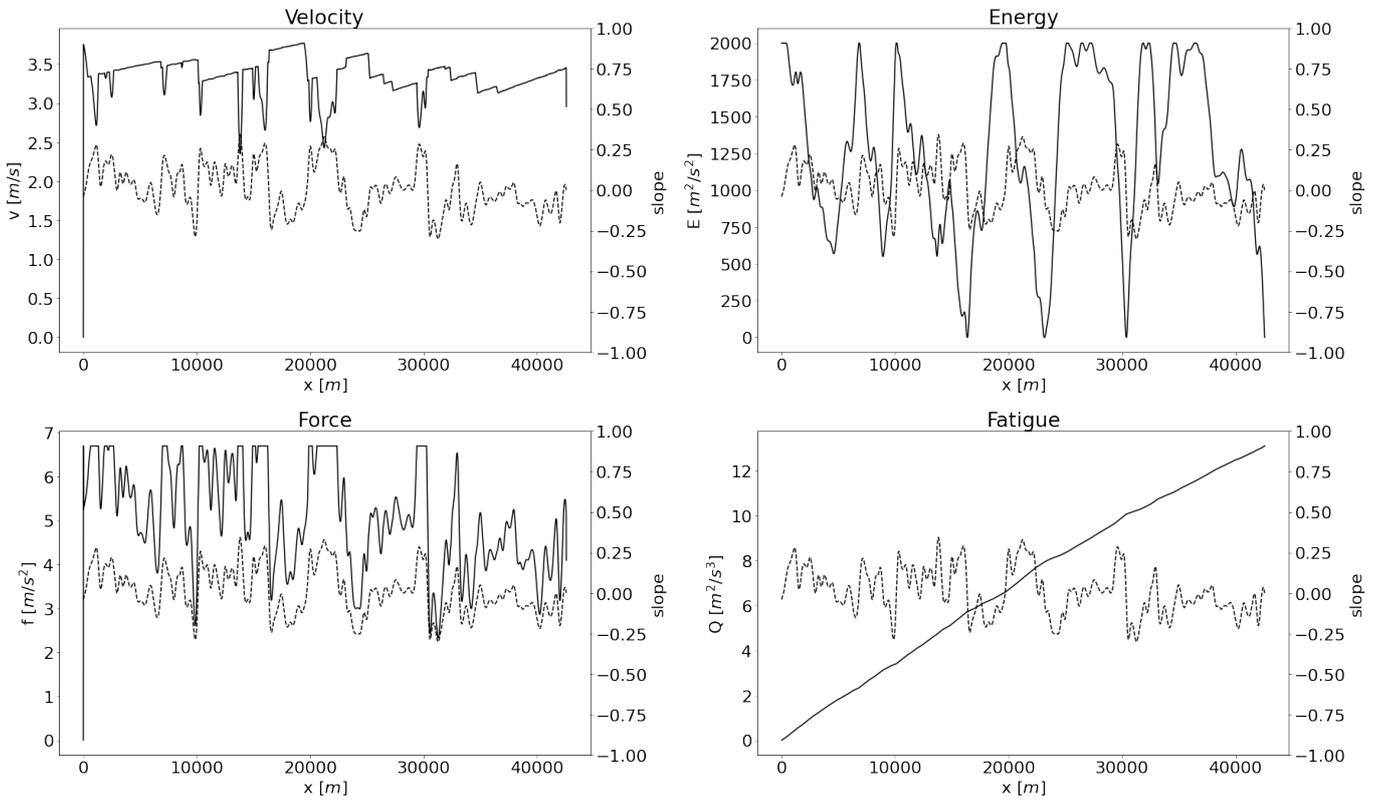}
\caption{Values of model variables for Zegama Aizkorri Maraton.}
\label{fig:allFromZegama}
\end{figure}

\section{Conclusion}
A suitable generalization of the classical Keller's model can be successfully applied to provide an optimal running strategy even for long races on a varying and difficult terrain. The factors included in our extension are nutrition and fatigue, which during long races are crucial for optimizing runner's performance. We have shown that, when applied to real-world data of several routes of various categories, our model provides an accurate estimation of the finishing time. Taking into account the difficulty of the terrain, as well as unpredictable weather conditions, we can conclude that the model performs very well for many different scenarios. 

Since our computations are generic, the results obtained might be improved, for example, by taking into account the following points.
\begin{itemize}
\item Providing detailed information about the runner, such as the exact value of their $\dot{V}O2_{\max}$ and body mass instead of some generic values. 
\item Including the type of surface in the calculations. Although such data are hard to obtain, taking into account the impact of the pavement will make the model much more realistic. It can be achieved by making $K$ the function of $x(t)$ or adding the correction factor to the velocity equation. 
\end{itemize}

\bibliographystyle{plain} 
\bibliography{bibliografia}

\end{document}